\documentclass[a4paper,12pt]{article}
\usepackage{amssymb}
\usepackage{amsthm}
\usepackage{latexsym,array}
\usepackage{amsthm}
\usepackage{amsmath, mathrsfs}
\usepackage[english]{babel}
\usepackage{color}
\usepackage[mathscr]{euscript}
\usepackage{hyperref}
\usepackage[left=2.5cm,right=2.5cm,top=2cm,bottom=2cm]{geometry}
\usepackage{cite}
\hypersetup{
    colorlinks=true,
    linkcolor=blue,
    filecolor=magenta,      
    urlcolor=blue,
    citecolor=red,
    }

\newtheorem{theorem}{Theorem}[section]
\newtheorem{definition}[theorem]{Definition}

\newtheorem{proposition}[theorem]{Proposition}

\newtheorem{remark}[theorem]{Remark}

\newcommand{\R}{\mathbb{R}}

\newcommand{\D}{\mathbb{D}}

\newcommand{\bk}{\mathbf{k}}
\newcommand{\be}{\mathbf{e}}

\newcommand{\sP}{\mathcal{P}}

\newcommand{\sG}{\mathfrak{G}}
\newcommand{\sB}{\mathscr{B}}
\newcommand{\sR}{\mathscr{R}}
\newcommand{\sJ}{\mathscr{J}}

\newcommand{\fLog} {\operatorname{Log}}%LogHyperbolic
\providecommand{\abs}[1]{\lvert#1\rvert}%absolutevalue
\providecommand{\norm}[1]{\lVert#1\rVert}%norm

\begin{document}
\title{Lesche stability of the Shannon strong hyperbolic entropy and some hyperbolic extensions}
\small{
\author {Juan Adrián Ramírez-Belman$^{(1)\footnote{corresponding author}}$, Juan Bory-Reyes$^{(2)}$,\\ José Oscar González-Cervantes$^{(3)}$, Gamaliel Yafte Tellez-Sanchez$^{(4)}$}
\vskip 1truecm
\date{\small $^{(1)}$ SEPI, ESFM-Zacatenco-Instituto Polit\'ecnico Nacional. 07338, Ciudad M\'exico, M\'exico\\ Email: adrianrmzb@gmail.com \\$^{(2)}$ {SEPI, ESIME-Zacatenco-Instituto Polit\'ecnico Nacional. 07338, Ciudad M\'exico, M\'exico}\\Email: juanboryreyes@yahoo.com \\
$^{(3,4)}$ {Departamento de Matemáticas, ESFM-Instituto Polit\'ecnico Nacional. 07338, Ciudad M\'exico, M\'exico}\\Email: jogc200678@gmail.com, gtellez.wolf@gmail.com \\
}}
\maketitle
\begin{abstract}
In recent decades, several definitions of new entropy measures have been proposed, which expands the range of applications for this important tool. The present work focuses on the extension of the classical Shannon entropy to the hyperbolic number plane $\D$ with the notion of valued hyperbolic probability. It is shown that the Shannon strong hyperbolic entropy over a discrete hyperbolic probability distribution $(\rho_{1}, \ldots,\rho_{N})$ can be established by the action of the hyperbolic derivative on the generating function $\sum_{s=1}^{N}{\rho}_{s}^{-\xi}$ with respect to the hyperbolic variable $\xi$ and then we tend $\xi$ to $-1_{\D}$. Furthermore, we prove that this hyperbolic extension possesses the Lesche stability property, also known as experimental robustness. Finally, we present some results on the hyperbolic extension of the Rényi entropy and hyperbolic extropy.
\vspace{0.5cm}

\noindent
%\textbf{Palabras Claves:} Entropía de Shannon, Entropía de Rényi, números hiperbólicos, derivada hiperbólica, estabilidad de Lesche.\\
\textbf{Keywords:} Shannon entropy, Rényi entropy, hyperbolic numbers, hyperbolic derivative, Lesche stability. \\
\textbf{MSC (2010):} 94A17, 30G35, 32A30, 32A10
\end{abstract}

\section{Introduction}

The term entropy comes from the Greek word {\em tropy} and means transformation or evolution. It was  Clausius, R., who named and developed the concept, see \cite{CR}. Later, Boltzmann, L., in his research \cite{BMa}, found a way to express this concept mathematically from the point of view of probability. Between 1890 and 1900, Boltzmann and other researchers developed the foundations of statistical physics, a theory closely related to the notion of entropy.

The Shannon entropy, as introduced by Shannon, C. E., originates in the field of information theory, see \cite{Sha48}. It translates to the degree of disorder in a system, a measure of uncertainty or randomness in a system or data set. It is a tool used in various areas of knowledge, such as information theory, cryptography, and data analysis, since its nature describes a quantification of the information necessary to describe an event within a set of possibilities.

Subsequently, extensions to the original work of Shannon emerged, resulting in several alternative entropy measures. Such is the case of Rényi, A., whose work in \cite{Ren61} by relaxing the Shannon third property (the additive property) and extended the Shannon entropy to a continuous parametric family of entropy measures. Examples of literature reviews on generalizations of entropy measures are \cite{ABH, LT}.

Recently, complex-valued functions have been introduced into the theory of uncertain information modeling. Ramot, D., along with Milo, R., Friedman, M., and Kandel, A., published a paper on complex fuzzy sets and fuzzy correlation terms, research conducted in \cite{RMFK}. In \cite{AUS}, Alkouri, A. U. M., and Salleh, A. R., propose intuitionistic complex-valued fuzzy sets, where the distance measure between complex-valued fuzzy sets is analyzed. Furthermore, theories related to complex-valued functions based on the mass function \cite{PLD, XFD, XFD1} have been proposed. Additionally, information carriers containing complex-valued signals include, for example, audio signals, image signals, and physiological signals, as discussed in the research of \cite{DSM, TJM,WLH}.

%En \cite{VNJ} se aborda la mecánica cuántica y se define la probabilidad complejo valuada, después se introduce en la teoría de decisión cuántica que describe la información, conocida también como probabilidad cuántica. El psicólogo matemático Busemeyer, J. R., cuya investigación se centra en los procesos cognitivos, la dinámica del juicio y la toma de decisiones humanas basado en modelos matemáticos. Bruza, P. D., investigador en el área de la cognición cuántica, estudia la cognición humana mediante un marco conceptual y estructuras matemáticas de la física cuántica. En \cite{BuB}, definen una probabilidad de valor complejo  al modelar y manejar incertidumbre de la información.

Uncertainty is unavoidable in decision-making systems applications. Therefore, researchers set out to measure and model uncertainty information using various tools described in \cite{FCP, YAE, LPD, XXW, DJi, LBU, JCD, FFL}. In decision-making systems, measuring uncertain information is an open topic, especially in information processing modeled on complex planes. In \cite{XYU}, Xiao, F., and Yue, X. G., define a complex entropy to measure the uncertainty of a complex probability distribution.

A fractional Tsallis entropy over a complex probability distribution is set at \cite{IRD}. Complex-valued neural networks are networks that are based on variables and parameters of complex values.

In \cite{GB21}, two extensions of the Shannon entropy are introduced via the hyperbolic logarithm over valued hyperbolic probability distributions: weak hyperbolic entropy and strong hyperbolic entropy.

Hyperbolic numbers, whose origin dates back to Cockle, J., who in \cite{C} devised his tessarines, can be considered a hybrid between real and complex numbers. In other words, they possess a nature similar to real numbers and exhibit properties related to complex numbers. For our purposes, we will use analogous notations or terminologies from \cite{Syk95, Syk13, GB21,ViVe}.

Hyperbolic numbers have applications in modeling genetic phenomena because they exhibit structural analogies to the molecular coding system and can be represented by bisymmetric matrices, which in special cases are doubly stochastic matrices. These matrices have various applications in linear programming, optimization, game theory, and algebraic biology \cite{PK}. Hyperbolic numbers are also related to two-dimensional spacetime geometry; in \cite{YAG}, they are used to describe the Lorentz magnitudes of special relativity in a spacetime plane.

In recent years, complex neural networks have had different areas of specialization and applications, such as the study by Akira, H., see \cite{AKi}. Hyperbolic neural networks were proposed by Buchholz, S., and Sommer, G., in \cite{BS} as an extension of real neural networks valued towards two dimensions. In \cite{NK}, Nitta, T., and Kuroe, Y., present a hyperbolic gradient operator and hyperbolic backpropagation learning algorithms. Furthermore, the hyperbolic environment has been used to represent solutions to three-dimensional equations and coordinates in special relativity, as mentioned in \cite{Syk95}.

Lesche, B., based in Brazil, demonstrated in \cite{Le} an interesting property about the Shannon entropy, which he called {\em stability}. However, in a private conversation, Tsallis, C., suggested to Lesche that he use the term {\em experimental robustness} instead of {\em stability} about the property he had demonstrated with respect to Shannon entropy. He agreed that this term better reflected the physical content of the mathematical property he himself had introduced. Furthermore, the expression {\em stability} could be confused with {\em thermodynamic stability}, which relates to the concavity of entropy, a completely different concept and an independent property. Therefore, the terms {\em Lesche stability} and {\em experimental robustness} are currently used interchangeably in the specialized literature, see \cite{TB}.

\section{Preliminaries} 
\subsection{The Shannon entropy}
Entropy can be understood in two ways: as a measure of uncertainty or as the amount of information needed to limit, eliminate, or reduce uncertainty.

\begin{definition}
If each of the $N$ states of a system $X$, denoted by $X_{s}, s=1,\ldots,N$ has a probability $p_s$. Then the {Shannon entropy} denoted by $S$, is written as:
\begin{equation}\label{Eq:E_S}
    S(p_1,\ldots, p_N)=-\sum_{s=1}^{N}p_{s}\log p_{s},
\end{equation}
where “ $\log$” denotes the natural logarithm.
\end{definition}
For a uniform probability distribution, we have $p_s = N^{-1}, N \in {\mathbb N}$, and the Shannon entropy reaches its maximum value $S_{max}:=S(\displaystyle\frac{1}{N},\ldots, \displaystyle\frac{1}{N}) = \log N$, resulting in the Boltzmann formula, except for a multiplicative factor, which is called the Boltzmann constant \cite{BMa}. In \cite{Sha48}, an axiomatization was established so that a function that satisfies the following three properties is called a {measure or function of entropy}.
\begin{enumerate}
    \item $S$ is continuous in each $p_{s}$.
    \item If all $p_{s}$ are equal, then $S$ is a monotonically increasing function.
    \item Given $Y$ another set of probabilities, then $S(XY)=S(X)+S(Y)$.
\end{enumerate}
%En \cite{Ren61} se demostró que la función logaritmo natural satisface los axiomas de entropía, de acuerdo al siguiente Lema.
%\begin{lemma}\label{Lm:AF_R}{\em \cite{E6}}.
 %   Sea $f:\N \to \R$ una función aditiva, esto es
  %  \begin{equation*}
        %f(nm)=f(n)+f(m),
   % \end{equation*}
    %para cualquier $n,m\in \N$, y
    %\begin{equation*}
     %   \lim_{n\to \infty}(f(n+1)-f(n))=0.
    %\end{equation*}
%Entonces, tenemos
 %   \begin{equation*}
  %      f(n)=c\log n,
   % \end{equation*}
%donde $c$ es una constante.
%\end{lemma}
The process of refining the axiomatics of the concept of entropy, by providing foundations for information theory, was considered by Shannon himself, where he obtained an important set of results detailing how the Shannon standard information measure is unsatisfactory for determining some aspects of shared information dependence.

\vspace{2mm}

On the other hand, after many years of extensive theoretical and applied research on the conception and application of entropy in countless fields of science, computer science, and engineering, it may be surprising that the entropy measure of a probability distribution has a dual (exterior-related) uncertainty measure, a complementary companion called “extropy”, as noted by \cite{LSA1, LSA2}, which has been defined as:
\begin{equation}\label{Eq:E_J}
    J(p_1,\ldots, p_N)=-\sum_{s=1}^{N}(1-p_{s})\log (1-p_{s}).
\end{equation}
As with entropy, the probability distribution that maximizes extropy is a uniform distribution, so that
\begin{equation}
J_{max}:=J(\displaystyle\frac{1}{N},\ldots, \displaystyle\frac{1}{N})=(N-1)\log\displaystyle\frac{N}{N-1}.
\end{equation}
Formally, the source of the duality of these two measures of uncertainty is based on the equations
\begin{equation}
J(p_1,\ldots, p_N) = \sum_{s=1}^{N}S(p_s;1-p_s) - S(p_1,\ldots, p_N)
\end{equation}
and symmetrically
\begin{equation}
S(p_1,\ldots, p_N) = \sum_{s=1}^{N}J(p_s;1-p_s) - J(p_1,\ldots, p_N),
\end{equation}
where
\begin{align*}
    S(p_s;1-p_s)=J(p_s;1-p_s)=-p_s \log p_s - (1 - p_s) \log(1 - p_s).
\end{align*}
Elementary properties that relate the Shannon concepts of entropy and extropy, detailed in \cite[Appendix A]{LSA1}, are summarized below.
\begin{proposition}\label{Prop:Relacion ExtyEnt}
Let $X$ be a discrete random variable with support $\{X_1,\ldots, X_N\}$ and with probability mass function (distribution) $\{p_1,\ldots, p_N\}$. If $N = 2$, so $X$ contains only two states, then $S((p_1,\ldots, p_N)) = J((p_1,\ldots, p_N))$, but while $X$ contains three or more states; $N \geq 3$ then $S((p_1,\ldots, p_N)) \geq J((p_1,\ldots, p_N))$.
\end{proposition}
In \cite{A1} the Shannon entropy formula was rewritten using the antiderivative of the logarithm function, so that the Shannon entropy \eqref{Eq:E_S} can be written as:
\begin{equation}\label{Eq:E_Slim}
    \begin{split}
    S(p_1,\ldots, p_N)= \lim_{t\to -1}\frac{d}{dt}\left(\sum_{s=1}^{N}p_{s}^{-t}\right)
		    =\lim_{t\to -1}-\sum_{s=1}^{N}p_{s}^{-t}\log p_{s}=-\sum_{s=1}^{N}p_{s}\log p_{s},
    \end{split}
\end{equation}
where $\displaystyle\sum_{s=1}^{N}p_{s}^{-t}$ it is called a generating function.

\subsection{The Rényi Entropy and Lesche Stability}
This section refers to Lesche, B., who proposed a stability criterion for the Rényi entropy, which is unstable for any value of the parameter $q$ except for $q=1$, see \cite{Le}.
\begin{definition}
    Given $q \in \R$ such that $q>0$, the Rényi entropy or order $q$ is defined as 
    \begin{equation}\label{Eq:R_E}
        R_{q}(\sP)=\frac{1}{1-q}\log \left(\sum_{s=1}^{N}p_{s}^{q}\right), \quad q\neq 1,
    \end{equation}
where $\sP=(p_{1},\ldots,p_{N})$ it is a discrete probability distribution, see \cite{RW4}.
\end{definition}
\begin{remark}
If the order of Rényi entropy tends to one, it coincides with the Shannon entropy. On the other hand, Rényi entropy reaches its maximum value when equiprobability is reached $p_{s}=\displaystyle\frac{1}{N}, s=1,\ldots,N$: $R_{q, max}=\log N$. 
\end{remark}
Two special cases of Rényi entropy are the following:
\begin{enumerate}
    \item The Hartley entropy, or Hartley function, is a measure of uncertainty introduced by the American engineer and pioneering researcher in the field of electronics, Hartley, R., in 1928. Which states that for $q=0$:
\begin{equation}\label{Eq:E_Hart}
    R_{0}(\sP)=\frac{1}{1-0}\log \sum_{s=1}^{N} p_{s}^{0}=\log N.
\end{equation}
The above is a primitive construction, since, as Kolmogorov and Rényi point out, the Hartley function can be defined without introducing any notion of probability.
\item The {Collision entropy} is defined as the Rényi entropy for $q=2$, and is given by
\begin{equation}\label{Eq:E_Col}
    R_{2}(\sP)=\frac{1}{1-2}\log \sum_{s=1}^{N}p_{s}^{2}=-\log \sum_{s=1}^{N} p_{s}^{2}.
\end{equation}
\end{enumerate}
In \cite{Le} it is shown that the experimental effort required to distinguish between two probability distributions $\sP=(p_{1},\ldots,p_{N})$ and $\sP'=(p_{1}',\ldots,p_{N}')$ with appreciable accuracy is related to the metric
\begin{equation}\label{Eq:Le_norm_real}
    \norm{\sP-\sP'}:=\sum_{s=1}^{N}\abs{p_{s}-p_{s}'}.
\end{equation}

\begin{definition} [Lesche Stability]\label{Def:L_Sta}
Let $\sP,\sP'$ be two probability distributions. The Lesche stability of the Rényi entropy states that for any $\varepsilon>0$ there exists a $\delta>0$ such that
\begin{equation}\label{Eq:L_Sta}
        \norm{\sP-\sP'}\leq \delta \Rightarrow \frac{\abs{R_{q}(\sP)-R_{q}(\sP')}}{R_{q,\max}}<\varepsilon.
\end{equation}
\end{definition}
The main motivation for this type of stability is to check whether there is an observable change that can be detected each time the probability distribution $\sP$ in a set of microstates is perturbed by an infinitesimal amount. This criterion has already been applied in \cite{A2, A3, AKS, Fe, KS, WS, Oi, TB} for some generalizations of Shannon entropy. The Lesche stability of Shannon entropy can be formulated as follows:
\begin{definition}
Let $S:([0,1]^{N},\norm{\cdot})\longrightarrow (\R,\abs{\cdot})$ the Shannon entropy function. Given two probability distributions $\sP, \sP'\in [0,1]^{N}$, it holds that
\begin{equation}\label{Eq:Est_Shan}
    \norm{\sP-\sP'}<\delta \Rightarrow \frac{\abs{S(\sP)-S(\sP')}}{\log N}<\varepsilon.
\end{equation}
\end{definition}

\subsection{Hyperbolic numbers}
Let us recall some basic facts about hyperbolic numbers, which can be found in more detail in \cite{EMDA, Syk95, Syk13, GB21, ViVe}. 

The set of hyperbolic numbers is defined by
\begin{align*}
\D:=\R[\bk]=\{a+\bk b: a,b\in \R,  \bk \not \in \R; \bk^2=1\},
\end{align*}
and form a commutative ring with zero divisors. Two nontrivial idempotent elements are given by
\begin{align*}
\be_{1}=\frac{1}{2}(1+\bk), \quad \be_{2}=\frac{1}{2}(1-\bk),
\end{align*} 
and are such that:
\begin{align*}
(\be_{1})^{2}=\be_{1},\quad
(\be_{2})^{2}=\be_{2}.
\end{align*}
Both elements cancel each other out
\begin{align*}
\be_{1}\cdot \be_{2}&=0.
\end{align*}
The set of divisors of zero is denoted as $\sG=\mathbb{R}\be_{1}\cup \mathbb{R}\be_{2}$ and $\sG_{0}=\sG\cup \{0\}$.
\begin{proposition}
The correspondence
\begin{equation}\label{Eq:Emb_RtoD}
    x\to x\be_{1}+x\be_{2}=:\widetilde{x}.
\end{equation}
embedded $\R$ into $\D$. In particular, the hyperbolic number one, denoted by $1_{\D}$, is given by
\begin{equation*}
    1_{\D}=1\be_{1}+1\be_{2},
\end{equation*}
\end{proposition}
the above is known as {\em idempotent representation}.

\noindent Hyperbolic numbers are equipped with a partial order relation $\preceq$ and defined as follows:
\begin{equation*}
    \alpha\preceq \beta \quad \text{if only if} \quad a_{1}\leq b_{1} \quad \text{and} \quad a_{2}\leq b_{2},
\end{equation*}
where $\alpha=a_{1}\be_{1}+a_{2}\be_{2}$ y $\beta=b_{1}\be_{1}+b_{2}\be_{2}$.

%It can be directly verified that this relation is reflexive, transitive, and antisymmetric, thus representing a good generalization for the total order in the real numbers.

The strict ordering for hyperbolic numbers is defined by the rule
\[ \alpha \prec \beta \quad \text{if only if} \quad a_1 < b_1 \quad \text{and}\quad a_2 < b_2.\]
Given $\alpha,\beta \in \D$ such that $\alpha\preceq \beta$ and $\beta - \alpha \not \in \sG$, the closed hyperbolic interval $[\alpha,\beta]_{\D}$ relative to the partial order $\preceq$ can be defined as the set
\begin{equation*}
    [\alpha,\beta]_{\D}:=\{\xi\in \D:\alpha\preceq \xi\preceq \beta\}.
\end{equation*}
The open hyperbolic interval $(\alpha,\beta)_{\D}$ relative to the partial order $\preceq$ can be defined as the set
\begin{equation*}
    (\alpha,\beta)_{\D}:=\{\xi\in \D:\alpha\prec \xi\prec \beta\}.
\end{equation*}
The set of positive hyperbolic numbers is defined by
\begin{equation*}
    \D^{+}:=\{\xi\in \D: 0 \prec \xi\}.
\end{equation*}
%y la multiplicación hiperbólica es cerrada en $\D^{+}$.
\begin{proposition}\label{Prop:Pot_hyp_2}
Let $\alpha=a_{1}\be_{1}+a_{2}\be_{2},\beta=b_{1}\be_{1}+b_{2}\be_{2}\in \D$, then the hyperbolic power can be expressed as
\begin{equation*}
\alpha^{\beta}=(a_{1})^{b_{1}}\be_{1}+(a_{2})^{b_{2}}\be_{2}.
\end{equation*}
\end{proposition}
In $\D$ the modulus $\abs{\cdot}_{\bk}$ of the hyperbolic number $\alpha=a_{1}\be_{1}+a_{2}\be_{2}$ is defined by
\begin{equation*}
\abs{\alpha}_{\bk}:=\abs{a_{1}}\be_{1}+\abs{a_{2}}\be_{2}\in \D^{+}.
\end{equation*}

\subsection{Hyperbolic derivative}
We now continue with the definition of the hyperbolic derivative according to \cite{ViVe, GB21}.
\begin{definition}
A hyperbolic function valued $F:X\longrightarrow \D$ defined on $X\subset \D$ is determined by the idempotent representation:
\begin{equation*}
    F(x)=F_{1}(x)\be_{1}+F_{2}(x)\be_{2}, \quad x\in X,
\end{equation*}
where $F_{1}, F_{2}$ these are real-valued functions called components of $F$.
\end{definition}
\begin{definition}
A hyperbolic metric space is defined as the pair $(X,D)$ where $X\subset \D$ is a non-empty set and $D$ is a positive-valued hyperbolic function such that for any $x,y\in X$ it satisfies that:
    \begin{enumerate}
        \item $D(x,y)=0$ if only if $x=y$.
        \item $D(x,y)=D(y,x)$.
        \item $D(x,z) \preceq D(x,y)+D(y,z)$.
    \end{enumerate}
\end{definition}
The most common hyperbolic metric space is $(\D,D_{\bk})$ where $D_{\bk}$ is defined for all $\xi, \chi \in \D$ such that
\begin{equation*}
    D_{\bk}(\xi, \chi):=\abs{x_{1}-y_{1}}\be_{1}+\abs{x_{2}-y_{2}}\be_{2},
\end{equation*}
where $\xi=x_{1}\be_{1}+x_{2}\be_{2}$ and $\chi=y_{1}\be_{1}+y_{2}\be_{2}$.
\begin{definition}
The function $F:(X,D_{1})\longrightarrow (X,D_{2})$ is said to be continuous at $\xi_{0}$, if for all $\varepsilon\in \D^{+}$, then there exists a $\delta \in \D^{+}$ such that for any $\xi\in \D$ with $D_{1}(\xi,\xi_{0})\prec \delta$, then
\begin{equation*}
    D_{2}(F(\xi),F(\xi_{0}))\prec \varepsilon.
\end{equation*}
\end{definition}
If the spaces $(X,D_{1}), (Y,D_{2})$ are equal to $(\D, D_{\bk})$ then hyperbolic continuity implies continuity $(\R^{2},d_{e})$ where $d_{e}$ is the standard Euclidean metric.
\begin{definition}
Let $F:(\D, D_{\bk})\longrightarrow (\D,D_{\bk})$. $F$ is said to be differentiable in $\xi\in \D$, if the limit exists
\begin{align}\label{Eq:Hyper_deriv}
F'(\xi)=\lim_{\substack{\psi\to 0\\\psi\not\in\sG_{0}}}\frac{F(\xi+\psi)-F(\psi)}{\psi}.
\end{align}
\end{definition}
In \cite{ViVe} it was shown that any differentiable hyperbolic function $F=u+\bk v$ in the basis $\{1,\bk\}$, where $u$ and $v$ are real-valued functions, satisfies the Cauchy-Riemann type equations
\begin{align*}
    \frac{\partial u}{\partial x}(\xi)=\frac{\partial v}{\partial y}(\xi), \quad \frac{\partial u}{\partial y}(\xi)=-\frac{\partial v}{\partial y}(\xi).
\end{align*}
Furthermore, if $F$ is rewritten as $F=F_{1}\be_{1}+F_{2}\be_{2}$ at the point $\xi=x_{1}\be_{1}+x_{2}\be_{2}$, then
\begin{equation*}
    F(\xi)=F_{1}(x_{1})\be_{1}+F_{2}(x_{2})\be_{2}
\end{equation*}
and
\begin{equation}\label{Eq:Hyp_der_par}
F'(\xi)=\frac{\partial F_{1}}{\partial x_{1}}(\xi)\be_{1}+\frac{\partial F_{2}}{\partial x_{2}}(\xi)\be_{2}=\frac{d F_{1}}{d x_{1}}(x_{1})\be_{1}+\frac{d F_{2}}{d x_{2}}(x_{2})\be_{2}.
\end{equation}
Now, the notion of a convex and concave hyperbolic function is constructed.
\begin{definition}
    Let $F:X\longrightarrow \D$ defined on $X\subset \D$ be a hyperbolic function valued with idempotent representation $F=F_{1}\be_{1}+F_{2}\be_{2}$ on a hyperbolic interval $[0,1_{\D}]_{\D}=[0,1]\be_{1}+[0,1]\be_{2}$. $F$ is said to be a convex hyperbolic function if, given the points $\xi=x_{1}\be_{1}+x_{2}\be_{2},\chi=y_{1}\be_{1}+y_{2}\be_{2}\in \D$ on $[0,1_{\D}]_{\D}$ such that $\xi\preceq\chi$ and for any $\lambda\in [0,1_{\D}]_{\D}$ such that $\lambda=\lambda_{1}\be_{1}+\lambda_{2}\be_{2}\in \D$, then
\begin{equation}\label{Eq:Convex_H}
    F((1-\lambda)\xi+\lambda\chi)\preceq (1-\lambda)F(\xi)+\lambda F(\chi).
\end{equation}
Similarly, if $\xi\succeq \chi$ is said to be a concave hyperbolic function, then
\begin{equation}\label{Eq:Concave_H}
    F((1-\lambda)\xi+\lambda\chi)\succeq (1-\lambda)F(\xi)+\lambda F(\chi).
\end{equation}
\end{definition}
\noindent The expansion of the inequality \eqref{Eq:Convex_H} is equivalent to
\begin{align*}
    &(F_{1}\be_{1}+F_{2}\be_{2})(((1\be_{1}+1\be_{2})-(\lambda_{1}\be_{1}+\lambda_{2}\be_{2}))\\ &\times(x_{1}\be_{1}+x_{2}\be_{2})+(\lambda_{1}\be_{1}+\lambda_{2}\be_{2})(y_{1}\be_{1}+y_{2}\be_{2})) \\
    &\preceq ((1\be_{1}+1\be_{2})-(\lambda_{1}\be_{1}+\lambda_{2}\be_{2}))(F_{1}\be_{1}+F_{2}\be_{2})(x_{1}\be_{1}+x_{2}\be_{2})\\ &+(\lambda_{1}\be_{1}+\lambda_{2}\be_{2})(F_{1}\be_{1}+F_{2}\be_{2})(y_{1}\be_{1}+y_{2}\be_{2}).
\end{align*}
So, \eqref{Eq:Convex_H} in idempotent representation is
\begin{align*}
    &\Big(F_{1}((1-\lambda_{1})x_{1}+\lambda_{1}y_{1})\leq (1-\lambda_{1})F_{1}(x_{1})+\lambda_{1}F(y_{1})\Big)\be_{1}+\\
    &\Big(F_{2}((1-\lambda_{2})x_{2}+\lambda_{1}y_{2})\leq (1-\lambda_{1})F_{2}(x_{2})+\lambda_{2}F(y_{2})\Big)\be_{2},
\end{align*}
where $F_{1},F_{2}$ are convex real functions respectively. According to the above, \eqref{Eq:Concave_H} also has an idempotent representation 
\begin{align*}
    &\Big(F_{1}((1-\lambda_{1})x_{1}+\lambda_{1}y_{1})\geq (1-\lambda_{1})F_{1}(x_{1})+\lambda_{1}F(y_{1})\Big)\be_{1}+\\
    &\Big(F_{2}((1-\lambda_{2})x_{2}+\lambda_{1}y_{2})\geq (1-\lambda_{1})F_{2}(x_{2})+\lambda_{2}F(y_{2})\Big)\be_{2},
\end{align*}
where $F_{1},F_{2}$ are concave real functions respectively. From Jensen's inequality of real numbers, it follows that if $X$ is a random variable and given $g(x)$ a convex function, then $g(E(X))\leq E(g(X))$ and if $g(x)$ is a concave function, then $g(E(X))\geq E(g(X))$. Furthermore, $g(x)$ is convex if and only if $-g(x)$ is concave.

\vspace{2mm}

In \cite{GB21}, a logarithm function is defined in the hyperbolic environment: $\fLog_{\D}$ (its use in the Shannon strong hyperbolic entropy is detailed in the following section). This logarithm is expressed in idempotent representation as
\begin{equation*}
    \fLog_{\D} \rho_{s}=\log p_{s,1}\be_{1}+\log p_{s,2}\be_{2}, \quad s=1,\ldots, N,
\end{equation*}
where $\log p_{s,1}$ and $\log p_{s,2}$ are concave real functions respectively. Then we can understand the hyperbolic extension of Jensen's inequality as an application of Jensen’s inequality to each component.

The following Theorem is a hyperbolic extension of L'Hôpital's Rule.
\begin{theorem}[Hyperbolic extension of L'Hôpital's rule]\label{Th:LHop_H}
    Let $F=F_{1}\be_{1}+F_{2}\be_{2}$ and $G=G_{1}\be_{1}+G_{2}\be_{2}$ be two hyperbolic functions valued at $\xi=x_{1}\be_{1}+x_{2}\be_{2}\in \D$ with idempotent representation. Suppose that \begin{align*}
        \lim_{\xi\to \xi_{0}}F(\xi)&=\left(\lim_{x_{1}\to a_{1}}F_{1}(x_{1})\right)\be_{1}+\left(\lim_{x_{2}\to a_{2}}F_{2}(x_{2})\right)\be_{2}\\&=0\be_{1}+0\be_{2},\\ \lim_{\xi\to \xi_{0}}G(\xi)&=\left(\lim_{x_{1}\to a_{2}}G_{1}(x_{1})\right)\be_{2}+\left(\lim_{x_{2}\to a_{2}}G_{2}(x_{2})\right)\be_{2}\\&=0\be_{1}+0\be_{2},
    \end{align*}
    where $\xi_{0}=a_{1}\be_{1}+a_{2}\be_{2}\in \D$. Let us also assume that there exists
    \begin{align*}
        \lim_{\xi\to\xi_{0}}\frac{F'(\xi)}{G'(\xi)}=\left(\lim_{x_{1}\to a_{1}}\frac{F_{1}'(x_1)}{G_{1}'(x_1)}\right)\be_{1}+\left(\lim_{x_{2}\to a_{2}}\frac{F_{2}'(x_2)}{G_{2}'(x_2)}\right)\be_{2},
    \end{align*}
    where $G'(\xi_{0})\not\in\sG_{0}$. Then there is
    \begin{align*}
        \lim_{\xi \to \xi_{0}}\frac{F(\xi)}{G(\xi)}=\left(\lim_{x_{1}\to a_{1}}\frac{F_{1}(x_1)}{G_{1}(x_1)}\right)\be_{1}+\left(\lim_{x_{2}\to a_{2}}\frac{F_{2}(x_2)}{G_{2}(x_2)}\right)\be_{2}
    \end{align*}
    and
    \begin{equation*}
        \lim_{\xi\to\xi_{0}}\frac{F(\xi)}{G(\xi)}=\lim_{\xi\to \xi_{0}}\frac{F'(\xi)}{G'(\xi)}.
    \end{equation*}
\end{theorem}
\begin{proof}
    Since the expressions can be written with the idempotent representation and since each of the inputs are real, then we apply L'Hôpital's Rule \cite[Chapter 11, Theorem 9]{Spiv} component-by-component together with differentiability. 
\end{proof}

\section{Hyperbolic Probability and Entropy}
\subsection{Hyperbolic probability}
In \cite{Alp17}, the following notion of a valued hyperbolic probability measure was introduced. This complies with the generalized axiom system of the Kolmogorov system, where the properties of non-negative hyperbolic numbers are highlighted.
\begin{definition}
Let $\sB=\{\rho_{1},\rho_{2}, \ldots, \rho_{N}\}$ a finite set of hyperbolic numbers in $[0,1_{\D}]_{\D}$ given by their idempotent representation $\rho_{s}=p_{s,1}\be_{1}+p_{s,2}\be_{2},\ s=1,\ldots, N$. We say that $\sB$ is a hyperbolic probability distribution if we have one of these conditions:
\begin{enumerate}
    \item $\displaystyle\sum_{s=1}^{N}\rho_{s}=1_{\D}$,
    \item $\displaystyle\sum_{s=1}^{N}\rho_{s}=1\be_{1}$,
    \item $\displaystyle\sum_{s=1}^{N}\rho_{s}=1\be_{2}$.
\end{enumerate}
\end{definition}
\begin{remark}
Cases 2 and 3 imply that $\rho_{s}=p_{s,1}\be_{1}$ for each $1\leq s\leq n$ or we have $\rho_{s}=p_{s,2}\be_{2}$, for each $1\leq s \leq n$. So $\rho_{s}$ turns out to be a divisor of zero.
\end{remark}
%De la noción de  probabilidad de espacio producto, existe una interpretación de probabilidades hiperbólicas valuadas. 
%\begin{definition}
%Una probabilidad hiperbólica valuada $\rho_{s}=p_{s,1}\be_{1}+p_{s,2}\be_{2}$ se identifica con
%\begin{equation}\label{Eq:UsAP_R}
 %   p_{s}:=\frac{p_{s,1}+p_{s,2}}{2},
%\end{equation}
%la cuál es la probabilidad usual acumulada en el caso real. \end{definition}

\subsection{Hyperbolic entropy}
In \cite{GB21}, the weak $S_{d}$ and strong $S_{f}$ hyperbolic extensions of Shannon entropy are introduced, where the holomorphic logarithm function is now replaced by the logarithm function in the theory of hyperbolic functions. We will take the strong hyperbolic extension of Shannon entropy as the Shannon strong hyperbolic entropy.

%donde es utilizada la relación de correspondencia hiperbólica \eqref{Eq:Emb_RtoD} del logaritmo real y las probabilidades acumuladas del tipo \eqref{Eq:UsAP_R}.
%\begin{definition}\label{Def:W_Sha} 
%Sea $\sB=(\rho_{1}, \rho_{2}, \ldots, \rho_{N})$ una distribución de probabilidad hiperbólica, se define la entropía hiperbólica débil asociada a $\sB$, como 
%\begin{equation}\label{Eq:W_HE}
%\begin{split}
 %    S_{d}(\sB)&=\sum_{s=1}^{N}-\rho_{s}\widetilde{\log}\,p_{s}\\&=\sum_{s=1}^{N}-\rho_{s}\left(\log(\frac{p_{s,1}+p_{s,2}}{2})\be_{1}+ \log(\frac{p_{s,1}+p_{s,2}}{2}) \be_{2}\right).
%\end{split}
%\end{equation}
%\end{definition}
%Aunque la entropía $ S_{d}$ puede ser claramente considerada como una generalización de la entropía de Shannon, cabe señalar que la expresión funcional de la entropía hiperbólica débil no cambia en esencia la medida de su antecesora. 
%Por ello, los autores propusieron una extensión hiperbólica fuerte de la entropía de Shannon, donde la función logaritmo holomorfo, es ahora sustituida por la bien definida función logaritmo en la teoría de funciones hiperbólicas.

\begin{definition}\label{Def:Stg_Sha}
Let $\sB=(\rho_{1}, \rho_{2}, \ldots, \rho_{N})$ a hyperbolic probability distribution. The Shannon strong hyperbolic entropy associated with $\sB$ is defined as 
\begin{equation}\label{Eq:Str_HE}
\begin{split}
    S_{f}(\sB)&=\sum_{s=1}^{N}-\rho_{s}\fLog_{\D}\rho_{s}\\ 
    &:=\sum_{s=1}^{N}-\rho_{s}(\log p_{s,1}\be_{1}+\log p_{s,2}\be_{2}),
\end{split}
\end{equation}
where $\rho_{s}=p_{s,1}\be_{1}+p_{s,2}\be_{2}$.
\end{definition}
\begin{remark}
Since the Shannon entropy measure, both hyperbolic extensions have fundamental properties that are legitimate as reasonable measures of choice:
\begin{enumerate}
\item Only if we are sure that the result of the measure disappears, otherwise it is positive.
\item Both $S_{d}$ and $S_{f}$ reach their maximum value and are equivalent to one of the values $\log (N)$, $\log (N)\be_{1}$, $\log(N)\be_{2}$ or also $\fLog_{\D}(N),\fLog_{\D}(N)\be_{1},\fLog_{\D}(N)\be_{2}$, respectively, if all the $\rho_{k}$ are equal (for example $\displaystyle\frac{1}{N}, \frac{1}{N}\be_{1},\frac{1}{N}\be_{2}$), which is the most uncertain situation.
\end{enumerate}
\end{remark}
\subsection{Rewriting the Shannon strong hyperbolic entropy}

We now aim to rewrite the Shannon strong hyperbolic entropy \eqref{Eq:Str_HE} in a different form, according to the following procedure: calculate the hyperbolic derivative of the hyperbolic generating function and the subsequent limit step.
\begin{definition}
Let $\sB=(\rho_{1}, \rho_{2}, \ldots, \rho_{N})$ be a hyperbolic probability distribution and $\xi=x_{1}\be_{1}+x_{2}\be_{2}\in \D^{+}$. The hyperbolic generating function is defined as
    \begin{align*}
        \sum_{s=1}^{N}\rho_{s}^{-\xi}
        &=\sum_{s=1}^{N}(p_{s,1}\be_{1}+p_{s,2}\be_{2})^{-(x_{1}\be_{1}+x_{2}\be_{2})}\\
        &=\sum_{s=1}^{N}(p_{s,1})^{-x_{1}}\be_{1}+\sum_{s=1}^{N}(p_{s,2})^{-x_{2}}\be_{2},
    \end{align*}
where the respective components are analogous to the real generating functions.
\end{definition}
\begin{proposition}
The Shannon strong hyperbolic entropy \eqref{Eq:Str_HE} can be rewritten, in analogy to the real case \eqref{Eq:E_Slim}.
\end{proposition}
\begin{proof}
The hyperbolic derivative of the hyperbolic generating function with respect to $\xi$ is then computed using the following equality \eqref{Eq:Hyp_der_par}:
\begin{align*}
    \left(\sum_{s=1}^{N}\rho_{s}^{-\xi}\right)'=&-\sum_{s=1}^{N}(p_{s,1})^{-x_{1}}\log p_{s,1}\be_{1}\\&-\sum_{s=1}^{N}(p_{s,2})^{-x_{2}}\log p_{s,2}\be_{2},
\end{align*}
where the component-by-component derivatives correspond to the generating function of the real Shannon entropy.

Then the limit is taken $\xi=x_{1}\be_{1}+x_{2}\be_{2} \longrightarrow -1_{\D}$. It follows that
\begin{align*}
    S_{f}(\sB)&=\lim_{\xi\to-1_{\D}}\left(\sum_{s=1}^{N}(p_{s,1})^{-x_{1}}\be_{1}+\sum_{s=1}^{N}(p_{s,2})^{-x_{2}}\be_{2}\right)'\\
    &=\lim_{x_{1}\to-1}\frac{\partial}{\partial x_{1}}\left(\sum_{s=1}^{N}(p_{s,1})^{-x_{1}}\right) \be_{1}+\lim_{x_{2}\to-1}\frac{\partial}{\partial x_{2}}\left(\sum_{s=1}^{n}(p_{s,2})^{-x_{2}}\right) \be_{2}	\\
    &=\lim_{x_{1}\to-1}\left(-\sum_{s=1}^{N}(p_{s,1})^{-x_{1}}\log p_{s,1}\right)\be_{1}+\lim_{x_{2}\to-1}\left(-\sum_{s=1}^{N}(p_{s,2})^{-x_{2}}\log p_{s,2}\right)\be_{2}\\
    &=\left(-\sum_{s=1}^{N}(p_{s,1})\log p_{s,1}\right)\be_{1}+\left(-\sum_{s=1}^{N}(p_{s,2})\log p_{s,2}\right)\be_{2}\\
    &=\sum_{s=1}^{N}-(p_{s,1}\be_{1}+p_{s,2}\be_{2})(\log p_{s,1}\be_{1}+\log p_{s,2}\be_{2}) \\
    &=\sum_{s=1}^{N}-\rho_{s}\fLog_{\D}(\rho_{s}).
\end{align*}
\end{proof}

\section{Lesche stability for the Shannon strong hyperbolic entropy}
This section focuses on the analysis of Lesche stability for the Shannon strong hyperbolic entropy.
\begin{theorem}
Let
$S_{f}:([0,1]_{\D}^{N},\norm{\cdot}_{\bk})\longrightarrow (\D,\abs{\cdot}_{\bk})$ be the Shannon strong hyperbolic entropy. Given two hyperbolic probability distributions $\sB, \sB'\in [0,1]_{\D}^{N}$, then
\begin{align*}
\norm{\sB-\sB'}_{\bk}\prec\delta \Rightarrow \frac{\abs{S_{f}(\sB)-S_{f}(\sB')}_{\bk}}{\widetilde{\log}N}\prec\varepsilon.
\end{align*}
\end{theorem}
\begin{proof}
Let $\sB=(\rho_{1}, \rho_{2}, \ldots, \rho_{N})$ y $\sB'=(\rho_{1}', \rho_{2}', \ldots, \rho_{N}')$ be two hyperbolic probability distributions in idempotent representation. Recall that  $\rho_s = p_{s,1}   \be_1 + p_{s,2}   \be_2 $ y $\rho_s' = p_{s,1}'   \be_1 + p_{s,2}'   \be_2 $ for each $s\in\{1,\dots, N\}$ and to simplify the notation denote
\begin{align*}
    \sB=\sP_{1}\be_{1}+\sP_{2}\be_{2}, \quad \sB'=\sP_{1}'\be_{1}+\sP_{2}'\be_{2},
\end{align*}
    where 
\begin{align*}
    \sP_{1}=(p_{1,1},\ldots,p_{N,1}), &\quad \sP_{1}'=(p_{1,1}',\ldots,p_{N,1}'),\\ 
    \sP_{2}=(p_{1,2},\ldots,p_{N,2}), &\quad \sP_{2}'=(p_{1,2}',\ldots,p_{N,2}').
\end{align*}
On the other hand, it is defined
\begin{align*}
    \norm{\sB-\sB'}_{\bk}:=\norm{\sP_{1}-\sP_{1}'}\be_{1}+\norm{\sP_{2}-\sP_{2}'}\be_{2}.
\end{align*}
We apply the metric \eqref{Eq:Le_norm_real}, then
\begin{align*}
    \norm{\sB-\sB'}_{\bk}&=\left(\sum_{s=1}^{N}\abs{p_{s,1}-p_{s,1}'}\right)\be_{1}+\left(\sum_{s=1}^{N}\abs{p_{s,2}-p_{s,2}'}\right)\be_{2}.
\end{align*}
Using the Lesche stability of Shannon entropy in the real case \eqref{Eq:Est_Shan} and given $\varepsilon = \varepsilon_1 \be _1 + \varepsilon_2\be _2 \succ 0 $, i.e.,  $\varepsilon_{1}, \varepsilon_{2}>0$, in each of the above summation there are $\delta_{1},\delta_{2}>0$, respectively such that
\[ 
\left(\sum_{s=1}^{N}\abs{p_{s,1}-p_{s,1}'}\right)< \delta_{1} \Rightarrow \abs{S(\sP_{1})-S(\sP_{1}')}< \varepsilon_{1},
\]
\[
\left(\sum_{s=1}^{N}\abs{p_{s,2}-p_{s,2}'}\right)< \delta_{2}, \Rightarrow \abs{S(\sP_{2})-S(\sP_{2}')}< \varepsilon_{2},
\] 
hence 
\begin{align*}
    \sum_{s=1}^{N}\abs{\sP_{1}-\sP_{1}'}\be_{1}+\sum_{s=1}^{N}\abs{\sP_{2}-\sP_{2}'}\be_{2}\prec\delta_{1}\be_{1}+\delta_{2}\be_{2}=:\delta
\end{align*}
implies
\begin{align*}
\frac{\abs{S_{f}(\sB)-S_{f}(\sB')}_{\bk}}{\widetilde{\log}N}&=\frac{\abs{S(\sP_{1})-S(\sP_{1}')}}{\log N}\be_{1}+\frac{\abs{S(\sP_{2})-S(\sP_{2}')}}{\log N}\be_{2}\\
    &\prec \varepsilon_{1}\be_{1}+\varepsilon_{2}\be_{2}=:\varepsilon.
\end{align*}
where $\varepsilon \succ 0$ and the Lesche stability is proven. 
\end{proof}

\section{The hyperbolic extension of Rényi entropy}
Motivated by \cite{GB21}, as we have already mentioned, the Shannon strong hyperbolic entropy is presented. With that impetus, the following entropy is introduced and establish some of its properties.
\subsubsection*{The hyperbolic extension of Rényi entropy of hyperbolic order $\alpha$} 
\begin{definition}[]
Let $\alpha=a_{1}\be_{1}+a_{2}\be_{2} \in \D$ such that $1_{\D}-\alpha \not\in\sG_{0}$ and let $\sB=(\rho_{1}, \rho_{2}, \ldots, \rho_{N})$ be a hyperbolic probability distribution. Then the hyperbolic extension of Rényi entropy of hyperbolic order $\alpha$, denoted by $\sR_{\alpha}$, is defined as:
\begin{equation}\label{Eq:RE_hyp_1}
\sR_{\alpha}(\sB):=\frac{1_{\D}}{1_{\D}-\alpha}\fLog_{\D} \left(\sum_{s=1}^{N}\rho_{s}^{\alpha}\right).
\end{equation}
\end{definition}
We see immediately that
\begin{align*}
    \sR_{\alpha}(\sB)&=\left(\frac{1\be_{1}+1\be_{2}}{(1\be_{1}+1\be_{2})-(a_{1}\be_{1}+a_{2}\be_{2})}\right)\fLog_{\D}\left(\sum_{s=1}^{N} (p_{s,1}\be_{1}+p_{s,2}\be_{2})^{a_{1}\be_{1}+a_{2}\be_{2}} \right)\\
    &=\left(\frac{1\be_{1}+1\be_{2}}{(1-a_{1})\be_{1}+(1-a_{2})\be_{2}} \right)\fLog_{\D}\left(\sum_{s=1}^{N}(p_{s,1}^{a_{1}})\be_{1}+\sum_{s=1}^{N}(p_{s,2}^{a_{2}})\be_{2}\right) \\
    &=\left(\frac{1}{1-a_{1}}\be_{1}+\frac{1}{1-a_{2}}\be_{2} \right)\left(\log\sum_{s=1}^{N}(p_{s,1}^{a_{1}})\be_{1}+\log\sum_{s=1}^{N}(p_{s,2}^{a_{2}})\be_{2}\right) \\
    &=\left(\frac{1}{1-a_{1}}\log\sum_{s=1}^{N}(p_{s,1}^{a_{1}})\right)\be_{1}+\left(\frac{1}{1-a_{2}}\log\sum_{s=1}^{N}(p_{s,2}^{a_{2}})\right)\be_{2}\\
    &=R_{a_{1}}(\sP_{1})\be_{1}+R_{a_{2}}(\sP_{2})\be_{2}.
\end{align*}
The above tells us that we can express $\sR_{\alpha}(\sB)$ in idempotent form, and since each component $R_{a_{1}}(\sP_{1})$, $R_{a_{2}}(\sP_{2})$ represents analogues of the Rényi entropy measure \eqref{Eq:R_E}, these are Lesche unstable. Consequently, the hyperbolic extension of Rényi entropy of hyperbolic order $\alpha$ \eqref{Eq:RE_hyp_1} is also Lesche unstable. 
\begin{proposition}
It immediately follows that 
\begin{equation*}\label{Eq:RE_hyp_2}
\sR_{1_{\D}}(\sB)=\lim_{\alpha\to 1_{\D}}\sR_{\alpha}(\sB)=-\sum_{s=1}^{n}\rho_{s}\fLog_{\D}(\rho_{s}).
\end{equation*}
\end{proposition}
\begin{proof}
Indeed, 
\begin{align*}
\sR_{1_{\D}}(\sB)&=(R_{1}(\sP_{1}))\be_{1}+(R_{1}(\sP_{2}))\be_{2}\\
&=\left(\lim_{a_{1}\to 1}R_{a_{1}}(\sP_{1})\right)\be_{1}+\left(\lim_{a_{2}\to 1}R_{a_{2}}(\sP_{2})\right)\be_{2}\\
&=\lim_{\alpha\to 1_{\D}}(R_{a_{1}}(\sP_{1})\be_{1}+R_{a_{2}}(\sP_{2})\be_{2})\\
&=\lim_{\alpha\to 1_{\D}}\sR_{\alpha}(\sB).
\end{align*}
Then
\begin{align*}
&\lim_{\alpha\to 1_{\D}}\sR_{\alpha}(\sB)=\left(\lim_{a_{1}\to 1}R_{a_{1}}(\sP_{1})\right)\be_{1}+\left(\lim_{a_{2}\to 1}R_{a_{2}}(\sP_{2})\right)\be_{2}  \\
&=\left(-\sum_{s=1}^{N}p_{s,1}\log p_{s,1}\right)\be_{1}+\left(-\sum_{s=1}^{N}p_{s,2}\log p_{s,2}\right)\be_{2} \\
&=-\sum_{s=1}^{N}\rho_{s}\fLog_{\D}(\rho_{s}).
\end{align*}
\end{proof}
Based on \cite{PLDR}, we examine the following:
\begin{proposition}
If $\alpha \to 1_{\D}$ together with the Theorem \ref{Th:LHop_H}, applied to the hyperbolic extension of Rényi entropy of hyperbolic order $\alpha$ \eqref{Eq:RE_hyp_1}, the Shannon strong hyperbolic entropy \eqref{Eq:Str_HE} is obtained:
\end{proposition}
\begin{proof}
Indeed,
\begin{align*}
    \lim_{\alpha\to 1_{\D}}\sR_{\alpha}(\sB)&=\lim_{\alpha\to 1_{\D}}\frac{\left(\fLog_{\D} \left(\sum_{s=1}^{N}\rho_{s}^{\alpha}\right)\right)'}{(1_{\D}-\alpha)'}\\
    &=\lim_{\alpha\to 1_{\D}}\frac{\left(\log \left(\sum_{s=1}^{N}p_{s,1}^{a_1}\right)\be_{1}+\log \left(\sum_{s=1}^{N}p_{s,2}^{a_2}\right)\be_{2}\right)'}{((1\be_{1}+1\be_{2})-(a_{1}\be_{1}+a_{2}\be_{2}))'}\\
    &=\lim_{a_{1}\to 1}\frac{\left(\log \left(\sum_{s=1}^{N}p_{s,1}^{a_1}\right)\right)'}{(1-a_{1})'}\be_{1}+\lim_{a_{2}\to 1}\frac{\left(\log \left(\sum_{s=1}^{N}p_{s,2}^{a_2}\right)\right)'}{(1-a_{2})'}\be_{2}\\
    &=\lim_{a_{1}\to 1}-\frac{\sum_{s=1}^{N}p_{s,1}^{a_{1}}\log p_{s,1}}{\sum_{k=1}^{N}p_{s,1}^{a_1}}\be_{1}+\lim_{a_{2}\to 1}-\frac{\sum_{s=1}^{N}p_{s,2}^{a_{2}}\log p_{s,2}}{\sum_{k=1}^{N}p_{s,2}^{a_2}}\be_{2}\\
    &=\left(-\sum_{s=1}^{N}p_{s,1}\log p_{s,1}\right)\be_{1}+\left(-\sum_{s=1}^{N}p_{s,2}\log p_{s,2}\right)\be_{2}\\
    &=-\sum_{s=1}^{N}\rho_{s}\fLog_{\D}(\rho_{s})=S_{f}(\sB).
\end{align*}
\end{proof}
Now, some properties of the hyperbolic extension of Rényi entropy of hyperbolic order $\alpha$ \eqref{Eq:RE_hyp_1} are as follows:

\begin{enumerate}
    \item Non-negativity of $\sR_{\alpha}(\sB)$, i.e., $\sR_{\alpha}(\sB)\succeq 0_{\D}$.

Let $\alpha=a_{1}\be_{1}+a_{2}\be_{2}\in \D$ such that $\alpha\succ 1_{\D}$ and $0_{\D}\preceq\rho_{s}\preceq 1_{\D}$ such that $\rho_{s}=p_{s,1}\be_{1}+p_{s,2}\be_{2}$, then $a_{1}>1$, $0\leq p_{s,1}\leq1$ and $a_{2}>1$, $0\leq p_{s,2}\leq1$. Then
\begin{equation*}
    \sum_{s=1}^{N}p_{s,1}^{a_1} \leq \sum_{s=1}^{N}p_{s,1}=1, \quad \sum_{s=1}^{N}p_{s,2}^{a_2} \leq \sum_{s=1}^{N}p_{s,2}=1.
\end{equation*}
It follows that
\begin{align*}
&\left(\log\sum_{s=1}^{N}p_{s,1}^{a_1} \right)\be_{1}+\left(\log\sum_{s=1}^{N}p_{s,2}^{a_2}\right)\be_{2}\preceq \left(\log\sum_{s=1}^ {N}p_{s,1}\right)\be_{1}+\left(\log\sum_{s=1}^ {N}p_{s,2}\right)\be_{2}=1_{\D},
\end{align*}
Then
\begin{align*}
&\fLog_{\D}\left(\sum_{s=1}^{N}\rho_{s}^{\alpha}\right)\preceq \fLog_{\D}(1_{\D})=0_{\D}.
\end{align*}
Therefore, $\sR_{\alpha}(\sB)\succeq 0_{\D}$. On the other hand, consider $\alpha\prec 1_{\D}$, then $a_{1}\leq 1$ and $a_{2}\leq 1$.
\begin{align*}
    \sum_{s=1}^{N}\rho_{s}^{\alpha}&=\left(\sum_{s=1}^{N}p_{s,1}^{a_1}\right)\be_{1}+\left(\sum_{s=1}^{N}p_{s,2}^{a_2}\right)\be_{2}\\ &\succeq \left(\sum_{s=1}^{N}p_{s,1}\right)\be_{1}+\left(\sum_{s=1}^{N}p_{s,2}\right)\be_{2}=\sum_{s=1}^{N}\rho_{s}=1_{\D}.
\end{align*}
It must satisfy the inequality
\begin{equation*}
    \fLog_{\D}\left(\sum_{s=1}^{N} \rho_{s}^{\alpha}\right)\succeq \fLog_{\D}(1_{\D})=0_{\D}.
\end{equation*}
Therefore, $\sR_{\alpha}\succeq 0_{\D}$.

    \item Non-increasing, i.e., $\sR_{\alpha}(\sB)$ if $\alpha=a_{1}\be_{1}+a_{2}\be_{2}\in \D$ increase.

First, the derivative of $\sR_{\alpha}$ with respect to $\alpha$ is calculated, then
\begin{align*}
    &\left(\frac{1_{\D}}{1_{\D}-\alpha}\fLog_{\D} \left(\sum_{s=1}^{N}\rho_{s}^{\alpha}\right)\right)'=\left(\frac{\fLog_{\D} \left(\sum_{s=1}^{N}\rho_{s}^{\alpha}\right)}{1_{\D}-\alpha}\right)'\\
    &=\frac{1_{\D}}{(1_{\D}-\alpha)^{2}}\left((1_{\D}-\alpha)\left(\fLog_{\D}\sum_{s=1}^{N}\rho_{s}^{\alpha}\right)'-\fLog_{\D}\sum_{s=1}^{N}\rho_{s}^{\alpha}(1_{\D}-\alpha)'\right)\\
    &=\frac{1_{\D}}{(1_{\D}-\alpha)^{2}}\left((1_{\D}-\alpha)\left(\frac{\sum_{s=1}^{N}\rho_{s}^{\alpha}\fLog_{\D}\rho_{s}}{\sum_{k=1}^{N}\rho_{k}^{\alpha}}\right)-\left(-\fLog_{\D}\sum_{s=1}^{N}\rho_{s}^{\alpha}\right)\right)\\
    &=\frac{1_{\D}}{(1_{\D}-\alpha)}\sum_{s=1}^{N}\frac{\rho_{s}^{\alpha}}{\sum_{k=1}^{N}\rho_{k}^{\alpha}}\fLog_{\D}\rho_{s}-\frac{1_{\D}}{(1_{\D}-\alpha)^{2}}\fLog_{\D}\frac{1_{\D}}{\sum_{k=1}^{N}\rho_{k}^{\alpha}}\\
    &=\frac{1_{\D}}{(1_{\D}-\alpha)}\sum_{s=1}^{N}\frac{\rho_{s}^{\alpha}}{\sum_{k=1}^{N}\rho_{k}^{\alpha}}\fLog_{\D}\rho_{s}-\frac{1_{\D}}{(1_{\D}-\alpha)^{2}}\fLog_{\D}\frac{\sum_{s=1}^{N}\rho_{s}}{\sum_{k=1}^{N}\rho_{k}^{\alpha}}\\
    &=\frac{1_{\D}}{(1_{\D}-\alpha)}\sum_{s=1}^{N}\frac{\rho_{s}^{\alpha}}{\sum_{k=1}^{N}\rho_{k}^{\alpha}}\fLog_{\D}\rho_{s}-\frac{1_{\D}}{(1_{\D}-\alpha)^{2}}\fLog_{\D}\frac{\sum_{s=1}^{N}\rho_{s}^{\alpha}\rho_{s}^{1_{\D}-\alpha}}{\sum_{k=1}^{N}\rho_{k}^{\alpha}}\\
    &=\frac{1_{\D}}{(1_{\D}-\alpha)^2}\left(\sum_{s=1}^{N}\frac{\rho_{s}^{\alpha}}{\sum_{k=1}^{N}\rho_{k}^{\alpha}}\fLog_{\D}\rho_{s}^{1_{\D}-\alpha}-\fLog_{\D}\frac{\sum_{s=1}^{N}\rho_{s}^{\alpha}\rho_{s}^{1_{\D}-\alpha}}{\sum_{k=1}^{N}\rho_{k}^{\alpha}}\right)
.\end{align*}
The above equality in idempotent representation is
\begin{align*}
    &=\left(\frac{1}{(1-a_{1})^{2}}\left(\sum_{s=1}^{N}\frac{p_{s,1}^{a_{1}}}{\sum_{k=1}^{N}p_{s,1}^{a_{1}}}\log p_{s,1}^{1-a_{1}}-\log\frac{\sum_{s=1}^{N}p_{s,1}^{a_{1}}p_{s,1}^{1-a_{1}}}{\sum_{k=1}^{N}p_{s,1}^{a_{1}}}\right)\right)\be_{1}\\
    &+\left(\frac{1}{(1-a_{2})^{2}}\left(\sum_{s=1}^{N}\frac{p_{s,2}^{a_{2}}}{\sum_{k=1}^{N}p_{s,2}^{a_{2}}}\log p_{s,2}^{1-a_{2}}-\log\frac{\sum_{s=1}^{N}p_{s,2}^{a_{1}}p_{s,2}^{1-a_{2}}}{\sum_{k=1}^{N}p_{s,2}^{a_{2}}}\right)\right)\be_{2}.
\end{align*}
Since every logarithm is a concave function and from Jensen's inequality of real numbers component-by-component, it can be observed that
\begin{align*}
    &\left(\sum_{s=1}^{N}\frac{p_{s,1}^{a_1}}{\sum_{k=1}^{N}p_{k,1}^{a_1}}\log p_{s,1}^{1-a_{1}}-\log\frac{\sum_{s=1}^{N}p_{s,1}^{a_1}p_{s,1}^{1-a_{1}}}{\sum_{k=1}^{N}p_{k,1}^{a_1}}\right)\be_{1}< 0\be_{1},\\
    &\left(\sum_{s=1}^{N}\frac{p_{s,2}^{a_2}}{\sum_{k=1}^{N}p_{k,2}^{a_2}}\log p_{s,2}^{1-a_{2}}-\log\frac{\sum_{s=1}^{N}p_{s,2}^{a_2}p_{s,2}^{1-a_{2}}}{\sum_{k=1}^{N}p_{k,2}^{a_2}}\right)\be_{2}< 0\be_{2}.
\end{align*}
That is, the Rényi entropy $R_{a_{1}}(\sP_{1})$ decreases if $a_{1}$ increases and $R_{a_{2}}(\sP_{2})$ decreases if $a_{2}$ increases, in other words, $\sR_{\alpha}(\sB)$ decreases if $\alpha=a_{1}\be_{1}+a_{2}\be_{2}\in \D$ increases.

    \item Concavity of $\sR_{\alpha}(\sB)$, i.e., if only if  $0_{\D}\prec\alpha\preceq1_{\D}$, $\sR_{\alpha}$ it is a concave hyperbolic function.

Let $\alpha=a_{1}\be_{1}+a_{2}\be_{2}\in \D$ such that $0_{\D}\prec\alpha\prec1_{\D}$. Then $0<a_{1}<1$ and $0<a_{2}<1$. It is true that for $s=1,\ldots,N$, thus
\begin{align*}
    \fLog_{\D}(\rho_{s})&=\log p_{s,1}\be_{1}+\log p_{s,2}\be_{2},\\ \rho_{s}^{\alpha}&=p_{s,1}^{a_1}\be_{1}+p_{s,2}^{a_2}\be_{2},
\end{align*}
where each $\log p_{s,1},\log p_{s,2}$ y $p_{s,1}^{a_1},p_{s,2}^{a_2}$ they are concave functions, then $\fLog_{\D}(\rho_{s})$ and $\rho_{s}^{\alpha}$ these are concave hyperbolic functions. It is known that,
\begin{align*}
    \sR_{\alpha}(\sB)&=\frac{1_{\D}}{1_{\D}-\alpha}\fLog_{\D}\sum_{s=1}^{N}\rho_{s}^{\alpha}\\ &=\left(\frac{1}{1-a_{1}}\log\sum_{s=1}^{N}(p_{s,1}^{a_{1}})\right)\be_{1}+\left(\frac{1}{1-a_{2}}\log\sum_{s=1}^{N}(p_{s,2}^{a_{2}})\right)\be_{2}.
\end{align*}
Hence, $\sR_{\alpha}(\sB)$ is a concave hyperbolic function. On the other hand, as stated earlier, the following holds
\begin{equation*}
    \displaystyle\lim_{\alpha\to 1_{\D}}\sR_{\alpha}(\sB)= -\sum_{s=1}^{N}\rho_{s}\fLog_{\D}\rho_{s},
\end{equation*}
Therefore, we also know that it is a concave hyperbolic function. Thus, $\sR_{1_{\D}}(\sB)$ is a concave hyperbolic function.

Let $\alpha=a_{1}\be_{1}+a_{2}\be_{2}\in \D$ such that $\alpha\succeq 1_{\D}$, then $a_{1}\geq 1$ and $a_{2}\geq 1$. Since $\fLog_{\D}(\rho_{s})$ it is a concave hyperbolic function because $\log p_{s,1}, \log p_{s,2}$, $s=1,\ldots, N$, they are concave functions. But now $\rho_{s}^{\alpha}$, $s=1,\ldots, N$ it is a convex hyperbolic function, since $p_{s,1}^{a_{1}},p_{s,2}^{a_{2}}$, $s=1,\ldots, N$, they are convex functions. Therefore $\sR_{\alpha}(\sB)$ it is neither a convex hyperbolic function nor a concave hyperbolic function. Finally, for $0<a_{1}\leq 1$ and $0<a_{2}\leq 1$, i.e., $0_{\D}\prec \alpha \preceq 1_{\D}$ it has to $\sR_{\alpha}(\sB)$ it is a concave hyperbolic function.
\end{enumerate}
As a consequence of the hyperbolic extension of Rényi entropy of hyperbolic order $\alpha$ \eqref{Eq:RE_hyp_1}, the following special cases are stated:
\begin{definition}[]
Let $\alpha=0\be_{1}+0\be_{2} \in \D$ and let $\sB=(\rho_{1}, \rho_{2}, \ldots, \rho_{N})$ be a hyperbolic probability distribution. The hyperbolic extension of Hartley entropy of hyperbolic order $\alpha=0_{\D}$, denoted by $\sR_{0_{\D}}$, is defined as
\begin{align*}
    \sR_{0_{\D}}(\sB):&=\frac{1_{\D}}{1_{\D}-0_{\D}}\fLog_{\D} \left(\sum_{s=1}^{N}\rho_{s}^{0_{\D}}\right)\\ &=\left(\frac{1}{1-0}\log\sum_{s=1}^{N}p_{s,1}^{0}\right)\be_{1}+\left(\frac{1}{1-0}\log\sum_{s=1}^{N}p_{s,2}^{0}\right)\be_{2}\\ &=R_{0}(\sP_{1})\be_{1}+R_{0}(\sP_{2})\be_{2},
\end{align*}
where the respective components are analogous to Hartley entropy \eqref{Eq:E_Hart}.
\end{definition}

\begin{definition}[]
Let $\alpha=2\be_{1}+2\be_{2} \in \D$ and let $\sB=(\rho_{1}, \rho_{2}, \ldots, \rho_{N})$ be a hyperbolic probability distribution. The hyperbolic extension of Collision entropy of hyperbolic order $\alpha=2_{\D}$, denoted by $\sR_{2_{\D}}$, is defined as
\begin{align*}
    \sR_{2_{\D}}(\sB):&=\frac{1_{\D}}{1_{\D}-2_{\D}}\fLog_{\D} \left(\sum_{s=1}^{N}\rho_{s}^{2_{\D}}\right)\\ &=\left(\frac{1}{1-2}\log\sum_{s=1}^{N}p_{s,1}^{2}\right)\be_{1}+\left(\frac{1}{1-2}\log\sum_{s=1}^{N}p_{s,2}^{2}\right)\be_{2}\\ &=R_{2}(\sP_{1})\be_{1}+R_{2}(\sP_{2})\be_{2},
\end{align*}
where the respective components are analogous to Collision entropy \eqref{Eq:E_Col}.
\end{definition}

\section{Final comments}
Finally, a first approach is introduced to two definitions that may mark the future development of this work.

\subsubsection*{The strong hyperbolic extension of Shannon extropy}
\begin{definition}[]
Let $\sB=(\rho_{1}, \rho_{2}, \ldots, \rho_{N})$ be a hyperbolic probability distribution. The strong hyperbolic extension of Shannon extropy associated with $\sB$, is defined as
\begin{equation*}
    J_{f}(\sB)=-\sum_{s=1}^{N}(1_{\D}-\rho_{s})\fLog_{\D}(1_{\D}-\rho_{s}).
\end{equation*}
\end{definition} 
In analogy to Proposition \ref{Prop:Relacion ExtyEnt}, the relation between Shannon strong hyperbolic entropy and the strong hyperbolic extension of Shannon extropy is presented.

If $N=2$, then $S_{f}(\sB)=J_{f}(\sB)$ holds. Indeed,
\begin{align*}
    S_{f}(\sB)&=\sum_{s=1}^{2}-\rho_{s}\fLog_{\D}\rho_{s}=-\rho_{1}\fLog_{\D}\rho_{1}-\rho_{2}\fLog_{\D}\rho_{2} \\ &= -(1_{\D}-\rho_{2})\fLog_{\D}(1_{\D}-\rho_{2})-(1_{\D}-\rho_{1})\fLog_{\D}(1_{\D}-\rho_{1})\\ 
    &=-\sum_{s=1}^{2}(1_{\D}-\rho_{s})\fLog_{\D}(1_{\D}-\rho_{s})=J_{f}(\sB).
\end{align*}
Let $N\geq 3$, that is, $\sB$ contains 3 or more positive hyperbolic components. Clearly, 
\begin{align*}
    S_{f}(\sB)=S(\sP_{1})\be_{1}+S(\sP_{2})\be_{2}
		\succeq J(\sP_{1})\be_{1}+J(\sP_{2})\be_{2}=J_{f}(\sB).
\end{align*}
which follows directly from using the Proposition \ref{Prop:Relacion ExtyEnt} component-by-component.

\subsubsection*{The strong hyperbolic extension of Rényi extropy with hyperbolic order $\alpha$}
\begin{definition}[]
Let $\alpha=a_{1}\be_{1}+a_{2}\be_{2} \in \D$ such that $1_{\D}-\alpha \not\in\sG_{0}$ and let $\sB=(\rho_{1}, \rho_{2}, \ldots, \rho_{N})$ be a hyperbolic probability distribution. The strong hyperbolic extension of Rényi extropy with hyperbolic order $\alpha$, denoted by $\sJ_{\alpha}$, can be introduced in analogy with \cite{LX} as:
\begin{align*}
\sJ_{\alpha}(\sB):&=\frac{1_{\D}}{1_{\D}-\alpha}[-(N-1_{\D})\fLog_{\D}(N-1_{\D}) +(N-1_{\D})\fLog_{\D}\left(\sum_{s=1}^{N}(1_{\D}-\rho_{s})^{\alpha}\right)].
\end{align*}
\end{definition}

\section*{Statements}
\subsection*{Financial Support} This work was partially supported by Instituto Polit\'ecnico Nacional  within the framework of the support programs of the Postgraduate Studies and Research Section.
\subsection*{Conflict of Interest} The authors declare that they have no conflict of interest related to this manuscript.
\subsection*{Authors Contributions} All authors made substantial contributions to the conception and design of the study. All participated in the drafting of the manuscript or its critical revision.

\subsection*{ORCID}
\noindent
Juan Adrián Ramírez-Belman: https://orcid.org/0009-0008-0873-8057 \\
Juan Bory-Reyes: https://orcid.org/0000-0002-7004-1794\\
Jos\'e Oscar Gonz\'alez-Cervantes: https://orcid.org/0000-0003-4835-5436\\
Gamaliel Yafte Tellez-Sanchez: https://orcid.org/0000-0002-6387-5336

\end{document}